%% file: Combinatorialproof.tex
\documentclass[12pt]{article}
\usepackage{amsfonts}
\usepackage{amsmath}
\usepackage{amssymb}
\usepackage[mathscr]{eucal}
\usepackage{fullpage}
\usepackage{times}
\usepackage[usenames]{color}
\usepackage{hyperref}

\def\eod{\vrule height 6pt width 5pt depth 0pt}
\newenvironment{proof}{\noindent {\bf Proof:} \hspace{.2em}}
                      {\hspace*{\fill}{\eod}}

\newcommand{\ceil}[1]{\left\lceil #1 \right\rceil}
\newcommand{\floor}[1]{\left\lfloor #1 \right\rfloor}

\newcommand{\altr}{ \mathrm{altruns}}
\newcommand{\SR}{ \mathrm{SgnAltrun}}

\newcommand{\SAS}{ \mathrm{SgnAltseq}}
\newcommand{\SSS}{\mathfrak{S}}
\newcommand{\AS}{ \mathrm{Altseq}}
\newcommand{\AAA}{\mathcal{A}}
\newcommand{\BB}{\mathfrak{B}}
\newcommand{\DD}{\mathfrak{D}}
\newcommand{\ZZ}{\mathbb{Z}}

\newcommand{\inv}{\mathrm{inv}}
\newcommand{\as}{\mathrm{as}}

\newcommand{\Sf}{\mathrm{Sgn\_flip}}
\newcommand{\Negs}{\mathrm{Negs}}

\newcommand{\nmhalf}{\lfloor (n-1)/2\rfloor}
\newcommand{\ol}[1]{\overline{#1}}

\newtheorem{theorem}{Theorem}

\newtheorem{remark}[theorem]{Remark}
\newtheorem{example}[theorem]{Example}

\newtheorem{lemma}[theorem]{Lemma}

\newcommand{\comment}[1]{}

\begin{document}
\title{On the Alternating runs polynomial in type B and D Coxeter Groups}

\author{Hiranya Kishore Dey\\
Department of Mathematics\\
Indian Institute of Technology, Bombay\\
Mumbai 400 076, India.\\
email: hkdey@math.iitb.ac.in
\and
Sivaramakrishnan Sivasubramanian\\
Department of Mathematics\\
Indian Institute of Technology, Bombay\\
Mumbai 400 076, India.\\
email: krishnan@math.iitb.ac.in
}

\maketitle

\begin{abstract}
Let $R_n(t)$ denote the polynomial enumerating alternating runs in 
the symmetric group $\SSS_n$.  Wilf showed that $(1+t)^m$ divides 
$R_n(t)$ where $m = 
\floor{(n-2)/2}$.  B{\'o}na recently gave a group action based 
proof of this fact.  In this work, we give a group action based 
proof for type B and type D analogues of this result.  Interestingly,
our proof gives a group action on the positive parts $\BB_n^{\pm}$
and $\DD_n^{\pm}$ and so we get refinements of the result to the
case when summation is over $\BB_n^{\pm}$ and $\DD_n^{\pm}$.  
We are unable to get a group action based proof of Wilf's
result when summation
is over the alternating group $\AAA_n$ and over $\SSS_n - \AAA_n$,
but using other ideas, give a different proof.
We give similar results to the polynomial which enumerates 
{\sl alternating sequences} in $\AAA_n$, $\SSS_n - \AAA_n$,
 $\BB_n^{\pm}$ and $\DD_n^{\pm}$.
As a corollary, we get moment type identities for coefficients
of such polynomials.
\end{abstract}

\section{Introduction}
\label{sec:intro}
For a positive integer $n$, let $[n] = \{1,2,\ldots,n\}$.  Let $\SSS_n$ 
denote the symmetric group on $[n]$.  
For an index $i$ with $2 \leq i \leq n-1$, 
we say that $\pi \in \SSS_n$ changes direction at $i$, if either
$\pi_{i-1} < \pi_i > \pi_{i+1}$ or if 
$\pi_{i-1} > \pi_i < \pi_{i+1}$. 
We say that $\pi$ has $k$ alternating runs, denoted by $\altr(\pi) = k$,
if $\pi$ has $k-1$ indices where it changes direction. For example, 
the permutation $\pi=1,3,2,4,6,5 \in \SSS_6$ has $4$ alternating runs.
Consider the polynomial
\begin{equation*}
R_n(t)= \sum_{\pi \in \SSS_n}t^{\altr(\pi)} = \sum_{k=1}^{n-1}R_{n,k}t^k.
\end{equation*}	

Andr{\'e} in \cite{Andrealtrun} started the study of permutations enumerated
by its number of alternating runs. 
Canfield and Wilf in \cite{canfieldwilfalternatingrun}, 
Stanley in \cite{stanleyalternatingsequence} and later 
Ma in \cite{maaltrun} give explicit formulae for $R_{n,k}$.
Wilf in \cite{wilfaltrun}, showed the following result
about the exponent of $(1+t)$ that divides $R_n(t)$. 

\begin{theorem}[Wilf]
\label{thm:divisibilty_of_alternatingrunpolynomial_by_highpower_oneplust}
For positive integers $n \geq 4$, the polynomial $R_n(t)$ is divisible by 
$(1 + t)^m$, where $m = \floor{(n -2)/2}$. 
\end{theorem}

Wilf's proof depends on a relation between 
the Eulerian polynomial and $R_n(t)$. Later, B{\'o}na and  Ehrenborg 
in \cite{bonaehrenborglogconcavity} gave an inductive proof of Theorem 
\ref{thm:divisibilty_of_alternatingrunpolynomial_by_highpower_oneplust}. 
Recently, B{\'o}na in \cite{bonaaltrun} gave a 
proof 
of Theorem
\ref{thm:divisibilty_of_alternatingrunpolynomial_by_highpower_oneplust}
based on group actions.  Inspired by B{\'o}na's proof, in this paper, 
we give a group action based proof of type B and type D counterparts
of Theorem 
\ref{thm:divisibilty_of_alternatingrunpolynomial_by_highpower_oneplust}.

Let $\BB_n$ be the set of permutations of 
$[\pm n] = \{\pm 1,\pm 2,\ldots,\pm n\}$ satisfying $\pi(-i)= - \pi(i)$. 
$\BB_n$ is known as the {\sl hyperoctahedral} group or the group 
of signed permutations on $[n]$.  For $\pi \in \BB_n$ and an 
index $1 \leq i \leq n$,  we denote $\pi(i)$ alternatively 
as $\pi_i$ and for $1 \leq k \leq n$, we denote $-k$ as 
$\overline{k}$. 
For $\pi = \pi_1,\pi_2,\ldots,\pi_n \in \BB_n$, define 
$\Negs(\pi) = \{\pi_i : i > 0, \pi_i < 0 \}$ as the set of 
elements which occur in $\pi$ with a negative sign. Let 
$\DD_n\subseteq \BB_n$ denote the subset consisting of those 
elements of $\BB_n$ which have an even number of negative entries. 
$\DD_n$ is referred to as the {\sl demihyperoctahedral} group on 
$[n]$.  For $\pi \in \BB_n$, let $\pi_0=0$.  
For an index $i$ with  $1 \leq i \leq n-1$,
we say that $\pi$ changes direction at $i$ if either
$\pi_{i-1} < \pi_i > \pi_{i+1}$ or if 
$\pi_{i-1} > \pi_i < \pi_{i+1}$. 
We say that $\pi \in \BB_n$ has $k$ alternating runs, 
denoted $\altr_B(\pi) = k$ if $\pi$ has 
$k-1$ indices where it changes direction.  
For example, the permutation 
$5,1,4,\overline{3}, \overline{6},2$ has  $\altr_B(\pi)=5.$ 
In $\DD_n$, we have the same definition of alternating runs, 
that is for $\pi \in \DD_n$, we have $\altr_D(\pi)=\altr_B(\pi)$.
Define 
$\BB_n^> = \{\pi \in \BB_n, \pi_1 > 0 \}$,
$\DD_n^> = \{\pi \in \DD_n, \pi_1 > 0$ \} and
$\BB_n^>-\DD_n^> = \{\pi \in \BB_n-\DD_n, \pi_1 > 0 \}.$
Further, define 
$\BB_n^< = \{\pi \in \BB_n, \pi_1 < 0 \}$,
$\DD_n^< = \{\pi \in \DD_n, \pi_1 < 0$ \} and
$\BB_n^<-\DD_n^< = \{\pi \in \BB_n-\DD_n, \pi_1 < 0 \}.$
Define the following polynomials:
\begin{equation*}
R_n^{B}(t) =  \sum_{\pi \in \BB_n}t^{\altr_B(\pi)}, \hspace{3 mm}
R_n^{D}(t) =  \sum_{\pi \in \DD_n}t^{\altr_D(\pi)}, \hspace{3 mm}
R_n^{B-D}(t) = \sum_{\pi \in \BB_n-\DD_n}t^{\altr_B(\pi)}, 
\end{equation*}
\begin{equation*}
R_n^{B,>}(t)  =  \sum_{\pi \in \BB_n^>}t^{\altr_B(\pi)}, \hspace{3 mm}
R_n^{D,>}(t)  =  \sum_{\pi \in \DD_n^>}t^{\altr_D(\pi)}, \hspace{3 mm}
R_n^{B-D,>}(t)  = \sum_{\pi \in \BB_n^>-\DD_n^>}t^{\altr_B(\pi)},
\end{equation*}
\begin{equation*}
R_n^{B,<}(t)  =  \sum_{\pi \in \BB_n^<}t^{\altr_B(\pi)}, \hspace{3 mm}
R_n^{D,<}(t)  =  \sum_{\pi \in \DD_n^<}t^{\altr_D(\pi)}, \hspace{3 mm}
R_n^{B-D,<}(t)  = \sum_{\pi \in \BB_n^<-\DD_n^<}t^{\altr_B(\pi)}.
\end{equation*}	

Zhao in \cite[Theorem 4.3.2]{zhaothesis} proved the following 
refinement of a type B analogue of Theorem 
\ref{thm:divisibilty_of_alternatingrunpolynomial_by_highpower_oneplust}.
See the paper by Chow and Ma 
\cite{chow-ma-counting-signed-perms-alt-runs} as well.
Later, Gao and Sun in \cite[Corollary 2.4]{Gaosunaltruntyped} 
considered the two polynomials $R_n^{D,>}(t)$ and $R_n^{B-D,>}(t)$ and
gave the following refinement of a type D analogue.
As the results are very similar, we have combined their respective 
results.
Though the form appears to be slightly different, Remark
\ref{rem:firstnegativeequalsfirstpositive-dtype} shows 
that Theorem
\ref{thm:divisibilty_of_alternatingrunpolynomial_by_highpower_oneplust_b}
implies a type B and type D counterpart of Theorem
\ref{thm:divisibilty_of_alternatingrunpolynomial_by_highpower_oneplust}.
Theorem
\ref{thm:divisibilty_of_alternatingrunpolynomial_by_highpower_oneplust_b}
refines and thus implies a type B and type D counterpart of Theorem
\ref{thm:divisibilty_of_alternatingrunpolynomial_by_highpower_oneplust}.

\begin{theorem}[Zhao, Gao and Sun]
\label{thm:divisibilty_of_alternatingrunpolynomial_by_highpower_oneplust_b}
For positive integers $n$, the polynomials $R_n^{B,>}(t)$, 
$R_n^{D,>}(t)$ and $R_n^{B-D,>}(t)$ are all divisible by 
$(1+t)^m$ where $m= \lfloor (n-1)/2\rfloor$.  
\end{theorem}

\begin{remark}
\label{rem:firstnegativeequalsfirstpositive-dtype}
Using the map $\Sf_1$ defined in Section 
\ref{sec:sgn_flip_sec}, it is easy to see that 
$R_n^{B,>}(t)= R_n^{B,<}(t)$ and therefore we have 
$R_n^{B}(t)= 2R_n^{B,>}(t)$.   Thus the 
polynomial $R_n^B(t)$ is divisible by $2(1+t)^m$. 

Using the same map, one can also see that 
$R_n^{D,<}(t)= R_n^{D,>}(t)$ when $n$ is even.  
When $n$ is odd, the same map gives us $R_n^{D,<}(t) =
R_n^{B-D,>}(t)$ and $R_n^{D,>}(t) = R_n^{B-D,<}(t)$.
Combining this with Theorem  
\ref{thm:divisibilty_of_alternatingrunpolynomial_by_highpower_oneplust_b},
we get that 
both $R_n^{D,<}(t)$ and  $R_n^{B-D,<}(t)$ 
are divisible by $(1+t)^m$ where $m={\nmhalf}$.  
Thus, when $n$ is even, the polynomial $R_n^D(t)$ 
is divisible by  $2(1+t)^m$ and when $n$ is
odd, $R_n^D(t)$ is divisible by $(1+t)^m.$  
\end{remark}

In Section \ref{Sec:proofofmainthm}, we give a group action based 
proof of Theorem 
\ref{thm:divisibilty_of_alternatingrunpolynomial_by_highpower_oneplust_b}.
Since $\SSS_n ,\BB_n$ and
$\DD_n$ are Coxeter groups they have a length function 
(see the book by Bj{\"o}rner and Brenti 
\cite{bjorner-brenti}) which we denote
by $\inv, \inv_B$ and $\inv_D$ respectively (see definitions
in Section \ref{sec:applns}).
Let $\BB_n^+ = \{ \pi \in \BB_n : \inv_B(\pi) \mbox{ is even}\}$ 
and let $\BB_n^- = \BB_n - \BB_n^+$.  Similarly, let $\DD_n^+ = \{ \pi 
\in \DD_n : \inv_D(\pi) \mbox{ is even}\}$ and let 
$\DD_n^- = \DD_n - \DD_n^+$.  Define 
$\BB_n^{>,\pm} = \BB_n^> \cap \BB_n^{\pm}$,
$\DD_n^{>,\pm} = \DD_n^> \cap \DD_n^{\pm}$ and
$(\BB_n-\DD_n)^{>,\pm} = (\BB_n - \DD_n)^> \cap \BB_n^{\pm}.$

Interestingly, we are able to get a group action based proof
of a refinement of Theorem 
\ref{thm:divisibilty_of_alternatingrunpolynomial_by_highpower_oneplust_b}
to $\BB_n^{>,\pm}$, $\DD_n^{>,\pm}$ and $(\BB_n - \DD_n)^{>,\pm}$.
Our refinements are Theorem \ref{thm:dbl_rstrct_div_b}  and
Theorem  \ref{thm:dbl_rstrct_div_d} and both are
proved in Section \ref{sec:applns}. 
As a corollary, we get moment type identities for the coefficients
of appropriate polynomials.  These are given in Theorem
\ref{thm:moment_id_dbl_rstrct_b} and Theorem
\ref{thm:moment_id_dbl_rstrct_d} respectively.

B{\'o}na's proof when directly applied to permutations in 
$\BB_n$ and $\DD_n$ also works for the type B and the 
type D case.  But, this straightforward extension to 
$\BB_n$, $\DD_n$, $\BB_n^>$ and $\DD_n^>$ does not seem to be 
an action over $\BB_n^{>,\pm}$.  Thus, it does not seem to give a proof 
of Theorem \ref{thm:dbl_rstrct_div_b} or 
Theorem \ref{thm:dbl_rstrct_div_d}.  
In Remark \ref{rem:on-bona}, we elaborate on this point.
We have given a refinement of Theorem 
\ref{thm:divisibilty_of_alternatingrunpolynomial_by_highpower_oneplust}
with summation over the alternating group $\AAA_n$ and $\SSS_n - \AAA_n$ 
in \cite[Theorem 3]{dey-siva-sgn-alt-runs-enum}, but our
proof there is not group action based.  We are unable to
get a group action that both preserves sign in $\SSS_n$ and gives
an exponent of roughly $n/2$ (see Remark \ref{rem:bona-not-for-An}).
It would be very interesting to get a group action based proof
of \cite[Theorem 3]{dey-siva-sgn-alt-runs-enum} that preserves sign.

B{\'o}na in \cite{bonaaltrun} also gave similar results about alternating
sequences.  For a permutation $\pi=\pi_1,\pi_2, \dots, \pi_n \in \SSS_n$,
an alternating subsequence of $\pi$ is a subsequence 
$\pi_{i_1}, \pi_{i_2}, \dots, \pi_{i_l}$ satisfying
$\pi_{i_1}> \pi_{i_2}< \pi_{i_3}> \dots \pi_{i_l}.$  Denote by $\as(\pi)$
the length of the longest alternating subsequence of $\pi$.  
For $\pi = \pi_1, \pi_2, \ldots, \pi_n \in \SSS_n$, define its number 
of inversions as $\inv(\pi) = |\{ 1 \leq i < j \leq n : \pi_i > \pi_j \}|$. 
Define the following polynomials
\begin{eqnarray*}
\AS_n(t) & = & \sum _{\pi \in \SSS_n} t^{\as(\pi)}, 
\hspace{10 mm}
\SAS_n(t)  =  \sum _{\pi \in \SSS_n}(-1)^{\inv(\pi)} t^{\as(\pi)}, \\
\AS_n^+(t) & = & \sum _{\pi \in \AAA_n} t^{\as(\pi)}, 
\hspace{24.5 mm}
\AS_n^-(t)  =  \sum _{\pi \in \SSS_n - \AAA_n} t^{\as(\pi)}.
\end{eqnarray*}

In \cite[Section 1.3.2]{bona-combin-permut},
B{\'o}na presents the following.
\begin{equation}
\label{eqn:reln_altruns-altseq}
\AS_n(t)= \displaystyle  \frac {(1+t)R_n(t)} {2}.
\end{equation}
An immediate consequence of \eqref{eqn:reln_altruns-altseq}
and Theorem \ref{thm:Bona_relation_between_altrun_altrsubseq} is
the following.

\begin{theorem}[B{\'o}na]
\label{thm:Bona_relation_between_altrun_altrsubseq}
For integers $n \geq 4$, the polynomial
$\AS_n(t)$ is divisible by $(1+t)^{\floor{n/2}}$.
\end{theorem}

B{\'o}na in \cite{bonaaltrun} also mentions that one can prove 
Theorem 
\ref{thm:Bona_relation_between_altrun_altrsubseq} directly by using 
the same group action he used to prove Theorem 
\ref{thm:divisibilty_of_alternatingrunpolynomial_by_highpower_oneplust}.
As mentioned above, we are unable to get a group action on $\AAA_n$
and on $\SSS_n - \AAA_n$.  We elaborate on it below.

\begin{remark}
\label{rem:bona-not-for-An}
We recall the group action defined by B{\'o}na. Given 
$\pi \in \SSS_n$ and an index  $i,$ consider the sequence 
$S_i = \pi_i, \dots, \pi_n$ of elements of $\pi$ starting
from the $i$-th position.  Let $c_i: \SSS_n \mapsto \SSS_n$ be the
map which replaces the $p$-th smallest element by the $p$-th largest 
element for all values of $p$ in the sequence $S_i$. 
Let $C_n= \{c_1, c_3, c_5, \dots c_t \}$  where $t = n - 1$ if $n$
is even and $t = n-2 $ if $n$ is odd.  B{\'o}na shows that 
$C_n$ acts on $\SSS_n$, but this action does not restrict to 
$\AAA_n$.  For instance, let $n=8$ and let $\pi= 2,1,4,3,6,5,8,7 \in
\SSS_8$.   It is easy to check that $\inv(\pi)=4$ and 
hence $\pi \in \AAA_8$ is an even permutation.   However, 
it can be checked  that $c_3(\pi)= 2,1,7,8,5,6,3,4$ has 
$\inv(c_3(\pi))=13$ and is thus an odd permutation.

Therefore we are unable to prove Theorem 
\ref{thm:divisibility_altseq_typea_pm} by using B{\'o}na's action.  
We can however get the following weaker
version.  It is easy to see that for a
permutation $\pi$, the permutation $c_i(\pi)$ will have the same 
sign if and only if $n-i+1 \equiv 0,1 $ (mod 4) (see 
\cite[Lemma 10]{siva-dey-gamma_positive_descents_alt_group} for
a proof).  Moreover, the key requirement of \cite[Lemma 2.4]{bonaaltrun}
is that indices $i$'s in $C_n$ should be non-consecutive.  
From these two conditions, we can choose $C_n= \{1,5,9,\dots,4k-3 \}$ 
when $n=4k$ or $n=4k+1$ and choose $C_n= \{3,7,11,\dots,4k-1 \}$ 
when $n=4k+2 $ or $n=4k+3$.
One can check that this $C_n$ acts on $\AAA_n$ and on
$\SSS_n - \AAA_n$. From this, we get 
that 
the polynomials $\AS_n^{\pm}(t)$ are divisible by $(1+t)^m$ 
where $m=\floor{n/4}$.
\end{remark}


The proof of our earlier result 
\cite[Theorem 3]{dey-siva-sgn-alt-runs-enum} gives refined 
counting which we use to give a counterpart of 
Theorem \ref{thm:Bona_relation_between_altrun_altrsubseq} for 
$\AAA_n$ and $\SSS_n - \AAA_n$. In Section \ref{sec:alt-seq-results}
of this paper, we prove the following.

\begin{theorem}
\label{thm:divisibility_altseq_typea_pm}
For positive integers $n$, let $m= \floor{n/2}$. When $n \equiv 0,1$ 
(mod 4), $\AS_n^{\pm}(t)$ is
divisible by $(1+t)^m$.   When $n \equiv 2,3$ (mod 4), $\AS_n^{\pm}(t)$ 
is divisible by $(1+t)^{m-1}$.
\end{theorem}

Clearly, when $n \equiv 0,1$ (mod 4), Theorem 
\ref{thm:divisibility_altseq_typea_pm} refines Theorem
\ref{thm:divisibilty_of_alternatingrunpolynomial_by_highpower_oneplust}
and when $n \equiv 2,3$ (mod 4), it is a near refinement as the
exponent falls short by 1.  
We move on to type B and type D counterparts of Theorem
\ref{thm:divisibility_altseq_typea_pm}.
For $\pi = \pi_0, \pi_1, \dots, \pi_n  \in \BB_n$
with $\pi_0=0$, an alternating subsequence of $\pi$ is a
subsequence $\pi_{i_1} > \pi_{i_2} < \pi_{i_3} > \dots ...\pi_{i_l}$.
By $\as_B(\pi)$, we denote the length of the longest alternating
subsequence of $\pi$. For $\pi \in \DD_n$,
we have the same definition of alternating subsequence
as in $\BB_n$ and so, let $\as_D(\pi)= \as_B(\pi)$.  
Define the following polynomials
\begin{eqnarray*}
   \AS_n^B(t) & = & \sum _{\pi \in \BB_n} t^{\as_B(\pi)}, 
   \hspace{10 mm}
   \AS_n^D(t)  =  \sum _{\pi \in \DD_n} t^{\as_D(\pi)}, \\
   \AS_n^{B,\pm}(t) & = & \sum _{\pi \in \BB_n^{\pm}} t^{\as_B(\pi)}, 
   \hspace{10 mm}
   \AS_n^{D,\pm}(t)  =  \sum _{\pi \in \DD_n^{\pm}} t^{\as_D(\pi)}. \\
\end{eqnarray*}
In Section \ref{sec:alt-seq-results} we give the 
following type B and type D counterpart of 
Theorem \ref{thm:Bona_relation_between_altrun_altrsubseq}.

\begin{theorem}
\label{thm:divisibility_altseq_typebd}
For positive integers $n$, the polynomials $\AS_n^B(t)$ are
divisible by $(1+t)^{\ceil{\frac{n}{2}}}$. For positive
integers $n$, the polynomial $\AS_n^D(t)$ is divisible by
$(1+t)^{\floor{\frac{n}{2}}}$. Further, for positive integers $n$,
the polynomials $\AS_n^{B, \pm}(t)$ and $\AS_n^{D, \pm}(t)$
are divisible by $(1+t)^{\floor{\frac{n}{2}}}$.
\end{theorem}

\section{The sign-flip map and its properties}
\label{sec:sgn_flip_sec}

Let $\pi = \pi_1, \pi_2, \dots, \pi_n \in \BB_n$. For 
$1 \leq k \leq n$, we define the \emph {sign-flip} map 
as follows:
$$\Sf_k(\pi_1, \pi_2, \dots, \pi_n)= 
\pi_1, \pi_2, \dots, \pi_{k-1}, \overline{\pi_{k}}, \dots, \overline{\pi_n} .$$
That is, $\Sf_k(\pi)$ flips the sign of all the elements 
of $\pi$ from $\pi_k$ onwards.

\begin{example}
Let $\pi=1,3,2,\overline{6},\overline{4},5 \in \BB_6$. 
Then, $\Sf_1(\pi)=\overline{1},\overline{3},\overline{2},6,4,\overline{5}.$ 
When $k=4$, we have $\Sf_4(\pi)= 1,3,2,6,4,\overline{5}$ and when
$k=5$, we have $\Sf_5(\pi)= 1,3,2,\overline{6},4,\overline{5}.$
\end{example}

We note some basic properties of the map $\Sf_i$.

\begin{lemma}
For positive integers $1 \leq i \leq n$, the map $\Sf_i$ is an involution.
\end{lemma}

\begin{proof}
Let $\pi= \pi_1, \pi_2, \dots, \pi_n \in \BB_n$. Clearly,
\begin{equation*}
	\Sf_i( \Sf_i(\pi))  =  \Sf_i (\pi_1, \pi_2, \dots,\pi_{i-1}, \overline{\pi_i},\dots, \overline{\pi_n} ) 
	 =  \pi
\end{equation*}
The proof is complete.
\end{proof}

We next show that for distinct indices $i$ and $j$, the maps $\Sf_i$ 
and $\Sf_j$ commute.

\begin{lemma}
\label{lem:commutativity}	
Let $1 \leq i < j \leq n$. Then, the maps $\Sf_i$ and $\Sf_j$ 
commute. That is, for $\pi \in \BB_n$, we have 
$\Sf_i( \Sf_j(\pi)) = \Sf_j( \Sf_i(\pi)).$
\end{lemma}

\begin{proof}
	We have
	\begin{eqnarray*}
		\Sf_i( \Sf_j(\pi)) & = & \Sf_i (\pi_1, \pi_2, \dots,\pi_{i-1}, \pi_i,\dots, \pi_{j-1}, \overline{\pi_{j}}, \dots, \overline{\pi_n} ) \\
		& = & \pi_1, \pi_2, \dots,\pi_{i-1}, \overline{\pi_i},\dots, \overline{\pi_{j-1}}, \pi_{j}, \dots,\pi_n \\
		& = & \Sf_j (\pi_1, \pi_2, \dots,\pi_{i-1}, \overline{\pi_i},\dots, \overline{ \pi_{j-1}}, \overline{\pi_{j}}, \dots, \overline{\pi_n} ) \\
		& = & \Sf_j(\Sf_i(\pi)).
	\end{eqnarray*}
	This completes the proof. 
\end{proof}

The proof of the following Lemma is along the same lines as B{\'o}na's proof 
of \cite[Proposition 2.3]{bonaaltrun}.  We however give a proof for 
the sake of completeness. 

\begin{lemma}
\label{lem:onemorealtrunthananother}
Let $n \geq 3$ and $2 \leq i \leq n-1$. Then, one of $\pi$ and 
$\Sf_i(\pi)$ has exactly one more alternating run than the other. 
\end{lemma}

\begin{proof}
Let $\pi= \pi_1, \dots, \pi_{i-2}, \pi_{i-1}, \pi_i, \pi_{i+1} 
\dots, \pi_n$.  Clearly, 
$$\Sf_i(\pi)= \pi_1, \dots, \pi_{i-2}, \pi_{i-1}, \overline{\pi_i}, \overline{\pi_{i+1}}, \dots, \overline{\pi_n}.$$
The number of alternating runs in the string 
$\pi_i,\pi_{i+1},\dots,\pi_n$ and its image under $\Sf_i$, that is in
the string $\overline{\pi_i}, \overline{\pi_{i+1}}, \dots, \overline{\pi_n}$ 
are clearly identical. Therefore, we only need to consider the
changes in the four-element string $\pi_{i-2},\pi_{i-1},\pi_i,\pi_{i+1}$ and 
$\pi_{i-2},\pi_{i-1}, \overline{\pi_i}, \overline{\pi_{i+1}}.$ 
We break them into $2^3=8$ cases as follows:
\begin{enumerate}
	
	\item  If $\pi_{i-2} < \pi_{i-1} < \pi_i < \pi_{i+1}$, then either $\pi_{i-2} < \pi_{i-1} < \overline{ \pi_i} > \overline{\pi_{i+1}}$ or $\pi_{i-2} < \pi_{i-1} > \overline{ \pi_i} > \overline{\pi_{i+1}}$.
	
	\item If $\pi_{i-2} < \pi_{i-1} < \pi_i > \pi_{i+1}$, then either $\pi_{i-2} < \pi_{i-1} < \overline{ \pi_i} < \overline{\pi_{i+1}}$ or $\pi_{i-2} < \pi_{i-1} > \overline{ \pi_i} < \overline{\pi_{i+1}}$.
	
	\item If $\pi_{i-2} < \pi_{i-1} > \pi_i < \pi_{i+1}$, then either $\pi_{i-2} < \pi_{i-1} > \overline{ \pi_i} > \overline{\pi_{i+1}}$ or $\pi_{i-2} < \pi_{i-1} < \overline{ \pi_i} > \overline{\pi_{i+1}}$.
	
	\item If $\pi_{i-2} < \pi_{i-1} > \pi_i > \pi_{i+1}$, then either $\pi_{i-2} < \pi_{i-1} < \overline{ \pi_i} < \overline{\pi_{i+1}}$ or $\pi_{i-2} < \pi_{i-1} > \overline{ \pi_i} < \overline{\pi_{i+1}}$.
	
		\item If $\pi_{i-2} > \pi_{i-1} < \pi_i < \pi_{i+1}$, then either $\pi_{i-2} > \pi_{i-1} < \overline{ \pi_i} > \overline{\pi_{i+1}}$ or $\pi_{i-2} > \pi_{i-1} > \overline{ \pi_i} > \overline{\pi_{i+1}}$.
		
		\item If $\pi_{i-2} > \pi_{i-1} < \pi_i > \pi_{i+1}$, then either $\pi_{i-2} > \pi_{i-1} < \overline{ \pi_i} < \overline{\pi_{i+1}}$ or $\pi_{i-2} > \pi_{i-1} > \overline{ \pi_i} < \overline{\pi_{i+1}}$.
		
		\item If $\pi_{i-2} > \pi_{i-1} > \pi_i < \pi_{i+1}$, then either $\pi_{i-2} > \pi_{i-1} > \overline{ \pi_i} > \overline{\pi_{i+1}}$ or $\pi_{i-2} > \pi_{i-1} < \overline{ \pi_i} > \overline{\pi_{i+1}}$.
		
		\item If $\pi_{i-2} > \pi_{i-1} > \pi_i > \pi_{i+1}$, then either $\pi_{i-2} > \pi_{i-1} < \overline{ \pi_i} < \overline{\pi_{i+1}}$ or $\pi_{i-2} > \pi_{i-1} > \overline{ \pi_i} < \overline{\pi_{i+1}}$.
	\end{enumerate}
In each of the eight cases above, the difference between
the number of alternating runs of the four-element string 
$\pi_{i-2},\pi_{i-1},\pi_i,\pi_{i+1}$ and 
$\pi_{i-2},\pi_{i-1}, \overline{\pi_i}, \overline{\pi_{i+1}}$ 
is exactly $1$. This completes the proof.   
 \end{proof}

\section{Proof of Theorem
\ref{thm:divisibilty_of_alternatingrunpolynomial_by_highpower_oneplust_b}}
\label{Sec:proofofmainthm}

Let $n \geq 3$ be an integer. When $n$ is even, let
$T_n= \{3,5,\dots,n-1\}$ and 
when $n$ is odd, let $T_n=\{2,4,\dots,n-1\}$. 
In both cases, let $M_n= \{ \Sf_i: i \in T_n \}.$
In either case, $M_n$ clearly consists of 
$m = \nmhalf$
pairwise commuting involutions.  Therefore, we get an action  
of the group $\ZZ_2^m$ on $\BB_n$ and $\DD_n$.  
For integers $n$ with $n \geq 3$, since $1 \not\in T_n$, 
we do not flip the sign of the first element and so $\ZZ_2^m$ 
also acts on $\BB_n^>$.  As the number of negative signs
is preserved, $\ZZ_2^m$ acts on $\DD_n^>$ as well.
With these set of generators, we note the following properties of 
the maps in $M_n$. 

\begin{lemma}
\label{lem:SfSfconsistency}
Let $i,j \in T_n$ with $i \neq j$ and suppose $\pi \in \BB_n$ 
is such that $\altr_B(\Sf_i(\pi))=\altr_B(\pi)+1$.
Then, $\altr_B(\Sf_j(\Sf_i(\pi))= \altr_B(\Sf_j(\pi))+1.$
\end{lemma}

\begin{proof}
Without loss of generality, let $i < j$. Let
$\pi= \pi_1, \dots, \pi_{i-2}, \pi_{i-1}, \pi_i, \pi_{i+1},  
\dots, \pi_j, \dots ,\pi_n$. We have 
\begin{eqnarray*}
\Sf_i(\pi) & = & \pi_1, \dots, \pi_{i-2}, \pi_{i-1}, 
\overline{\pi_i}, \overline{\pi_{i+1}}, \dots, \overline{\pi_{j-1}}, 
\overline{\pi_j}, \dots,  \overline{\pi_n}, \\
\Sf_j(\pi) & = & \pi_{1}, \dots,  \pi_{i-2}, \pi_{i-1}, \pi_i, 
\pi_{i+1},\dots, \pi_{j-1}, \overline{\pi_j}, \dots,  \overline{\pi_n}, \\
\Sf_j(\Sf_i(\pi)) & = & \pi_1, \dots, \pi_{i-2}, \pi_{i-1}, \overline{\pi_i}, \overline{\pi_{i+1}}, \dots, \overline{\pi_{j-1}}, \pi_j, \dots,  \pi_n.
\end{eqnarray*} 
As $\altr_B(\Sf_i(\pi))=\altr_B(\pi)+1$,
the string $\pi_{i-2}, \pi_{i-1}, \overline{\pi_i}, \overline{\pi_{i+1}}$ 
has one more alternating run than the string
$\pi_{i-2}, \pi_{i-1}, \pi_i, \pi_{i+1}.$
Thus,  we get 
$\altr_B(\Sf_j(\Sf_i(\pi)))= \altr_B(\Sf_j(\pi)) +1$. The proof is complete. 
\end{proof}

The action of $\ZZ_2^m$ on $\BB_n^>$ and $\DD_n^>$ clearly creates
orbits of size $2^m$.

\begin{lemma}
\label{lem:generating_function_in_orbit}
Let $O$ be any orbit of $\ZZ_2^m$ acting on $\BB_n^>$.
Then, for a non-negative integer $a$,  we have
$$ \sum_{\pi \in O} t^{\altr_B(\pi)} =  t^a(1+t)^m. $$ 
\end{lemma}

\begin{proof}
Let $\pi'$ be a permutation in the orbit $O$ with the minimum 
number of alternating runs, say $a$ many. Then for $i \in T_n$, 
we have $\altr_B(\Sf_i(\pi'))= \altr_B(\pi')+1$. 
For $i,j \in T_n$ with $i < j$, by Lemma \ref{lem:SfSfconsistency}, 
$\altr_B(\Sf_i(\Sf_j(\pi')))  =  \altr_B(\Sf_i(\pi'))+1= 
\altr_B(\pi')+2$. Therefore, there 
are $\binom{m}{2}$ permutations  in $O$ with $a+2$ alternating runs. 
Continuing this way, we get 
$\altr_B(\Sf_{i_1}(\Sf_{i_2}( \dots (\Sf_{i_k}(\pi'))))= \altr_B(\pi')+k$ 
for $i_1 < i_2 < \dots  < i_k$, where $i_1, i_2, \dots, i_k \in T_n$.
Clearly, there are $\binom{m}{k}$ many such elements. 
Summing over $k$, completes the proof. 
\end{proof}

We are now in a position to prove Theorem 
\ref{thm:divisibilty_of_alternatingrunpolynomial_by_highpower_oneplust_b}.

{\bf Proof of Theorem
\ref{thm:divisibilty_of_alternatingrunpolynomial_by_highpower_oneplust_b} :}
When $n \leq 2$, the exponent of $(1+t)$ is 0 and so there is nothing to prove.
Thus, let $n \geq 3$.   We first show the result for the type B case.
We write $\BB_n^>$ as a disjoint union of its orbits and 
apply Lemma \ref{lem:generating_function_in_orbit}.  Doing this, we  get
$R_n^{B,>}(t) = (1 + t)^m \sum_{\pi} t^ {\altr_B(\pi)} ,$
where the summation is over permutations $\pi \in \BB_n^>$ that have the
minimum $\altr_B(\pi)$ value in its orbit. 
This completes the proof for the type B case.

For the type D case, we make the same moves, while making the
following observation:  by our choice of the set $T_n$, we
assert that for each $i \in T_n$, we have 
$|\Negs(\pi)| \equiv |\Negs(\Sf_i(\pi))|$ (mod 2).  Therefore, all 
permutations $\pi$ of any orbit have the same value of 
$|\Negs(\pi)|$ (mod 2).   Thus, an orbit lies entirely in $\DD_n^>$ or  
entirely in $\BB_n^> - \DD_n^>.$  Decomposing $\DD_n^>$ and $\BB_n^> - \DD_n^>$ 
into orbits $O$ and summing $t^{\altr_D(\pi)}$ (and $t^{\altr_B(\pi)}$
respectively) over each $O$ completes the proof of the other two results.
{\hspace*{\fill}{\eod}}

\section{Refining Theorem 
\ref{thm:divisibilty_of_alternatingrunpolynomial_by_highpower_oneplust_b}
by taking parity into account}
\label{sec:applns}

As $\BB_n$ and $\DD_n$ are Coxeter groups, there is a natural notion of 
length, denoted by $\inv_B$ and $\inv_D$ respectively, associated to them. 
We first consider the type B case.  The following combinatorial 
definition is from Petersen's book 
\cite[Page 294]{petersen-eulerian-nos-book}:
\begin{equation}
\label{eqn:typeb-inv-defn}
\inv_B(\pi)=  | \{ 1 \leq i < j \leq n : \pi_i > \pi_j \} |+  
| \{ 1 \leq i < j \leq n : -\pi_i > \pi_j \}| +|\Negs(\pi)|.  
\end{equation}
For $\pi \in \BB_n$, we refer to $\inv_B(\pi)$ alternatively 
as its {\it length}.   Let $\BB^+_n \subseteq \BB_n$ 
denote the subset of even length elements of $\BB_n$ and 
let $\BB_n^- = \BB_n - \BB_n^+$.
Here, we use the notation $\BB_n^{\pm}$ to succintly denote 
both $\BB_n^+$ and $\BB_n^-$.  Similarly, we use terms like
$R_{n,k}^{B,>,\pm}$, 
and $R_{n,k}^{B,<,\pm}$ for brevity. 
%
%
Define the following polynomials enumerating $\altr_B$ over
various sets.
\begin{equation*}
	R_n^{B,>,\pm} (t)  =  \sum _{\pi \in \BB_n^{>,\pm}} 
t^ {\altr_B(\pi)}= \sum_{i=1}^n R_{n,i}^{B,>,\pm} t^i, \hspace{3 mm}
	R_n^{B,<,\pm} (t)  =  \sum _{\pi \in \BB_n^{<,\pm}} 
t^ {\altr_B(\pi)}=\sum_{i=1}^n R_{n,i}^{B,<,\pm} t^i.
\end{equation*}

Recall the sets $T_n$ from Section \ref{Sec:proofofmainthm}. We first 
show the following.  

\begin{lemma}
\label{lem:invb_preserved_under_action}
For positive integers $n$, and $\pi \in \BB_n$, we have 
$\inv_B(\Sf_k(\pi)) \equiv \inv_B(\pi) $ (mod 2) 
whenever $k \in T_n$.  
\end{lemma}

\begin{proof}
Observe that for any positive integer $n$, our choice of $T_n$ 
ensures the following: for all $k \in T_n$ and $\pi \in \BB_n$, the 
map $\Sf_k$ flips the sign of an even number of elements $\pi_i$'s
of $\pi$.  Flipping the sign of 
a single $\pi_i$ changes the parity $\inv_B(\pi)$
(see \cite[Lemma 3]{siva-sgn_exc_hyp}).  Thus, 
flipping the sign of an even number of $\pi_i$'s 
preserve the parity of $\inv_B(\pi)$. This completes the proof. 
\end{proof}

\begin{theorem}
\label{thm:dbl_rstrct_div_b}
For positive integers $n$, the polynomials $R_n^{B,>,\pm} (t)$ and
$R_n^{B,<,\pm} (t)$ 
are divisible by $(1+t) ^{m}$ where $m= \nmhalf$.
\end{theorem}
\begin{proof}
We first consider the polynomials $R_n^{B,>,+} (t)$. By 
Lemma \ref{lem:invb_preserved_under_action}, all permutations $\pi$ 
in an orbit have the same value of $\inv_B(\pi)$ (mod 2).
Thus, any orbit lies entirely in $\BB_n^{>,+}$ or entirely 
outside $\BB_n^{>,+}$.  Decomposing  $\BB_n^{>,+}$ as a disjoint
union of orbits $O$ and summing $t^{\altr_B(\pi)}$  over 
each $O$ proves that $R_n^{B,>,+} (t)$ is divisible by 
$(1+t) ^{m}$. In an identical manner, one can prove that 
$R_n^{B,>,-} (t)$ and $R_n^{B,<,\pm} (t)$ 
are divisible by $(1+t) ^{m}$.
\end{proof}

We need the following lemma which appears in the proof of a 
result of Chow and Ma 
\cite[Corollary 2]{chow-ma-counting-signed-perms-alt-runs}.  
We have paraphrased their
result, but from their proof, this change of form will be clear.

\begin{lemma}[Chow and Ma]
\label{lem:chow-ma-powers-alt}
If $(1+t)^m$ divides $f(t) = \sum_{i=0}^n f_i t^i$, then
for positive integers $k \leq m-1$, we have
$$1^k f_1 + 3^k f_3 + \cdots = 2^k f_2 + 4^k f_4 + \cdots.$$
\end{lemma}

From Lemma \ref{lem:chow-ma-powers-alt} and Theorem 
\ref{thm:dbl_rstrct_div_b}, we immediately get the 
following moment-type identity.  This refines a similar
identity involving $R_{n,k}^B$ and $R_{n,k}^D$ obtained
by combining Lemma 
\ref{lem:chow-ma-powers-alt} and Remark
\ref{rem:firstnegativeequalsfirstpositive-dtype}.

\begin{theorem}
\label{thm:moment_id_dbl_rstrct_b}
For $n \geq 2k+3$, we have
\begin{eqnarray*}
1^kR_{n,1}^{B,>, \pm}+3^kR_{n,3}^{B,>, \pm }+5^kR_{n,5}^{B,>,\pm}+\cdots 
& = & 
2^kR_{n,2}^{B,>,\pm}+4^kR_{n,4}^{B,>,\pm}+6^kR_{n,6}^{B,>,\pm}+\cdots, \\
1^kR_{n,1}^{B,<, \pm}+3^kR_{n,3}^{B,<, \pm }+5^kR_{n,5}^{B,<,\pm}+\cdots 
& = & 
2^kR_{n,2}^{B,<,\pm}+4^kR_{n,4}^{B,<,\pm}+6^kR_{n,6}^{B,<,\pm}+\cdots.
\end{eqnarray*}
\end{theorem}

We now move on to the type D case.  The following definition
for $\inv_D$ is from Petersen's book 
\cite[Page 302]{petersen-eulerian-nos-book}.
\begin{equation}
\label{eqn:inv-dtype-defn}
\inv_D(\pi) = \inv(\pi) + | \{1 \leq i < j 
\leq n:  -\pi_i > \pi_j \}|. 
\end{equation}
where $\inv(\pi) = |\{1 \leq i < j \leq n: \pi_i > \pi_j \}|$.
Let $\DD_n^+ = 
\{\pi \in \DD_n: \inv_D(\pi) \> \mbox{ is even} \}$  
and let $\DD_n^- = \DD_n - \DD_n^+$. 
For $\pi \in \DD_n$, recall that $\altr_D(\pi) = \altr_B(\pi)$.
Further, define the following polynomials 
enumerating $\altr_D$ over various sets.
\begin{equation*}
R_n^{D,>,\pm} (t)  = \sum _{\pi \in \DD_n^{>,\pm}} 
t^ {\altr_D(\pi)}= \sum_{i=1}^n R_{n,i}^{D,>,\pm} t^i,  \hspace{3 mm}
R_n^{D,<,\pm} (t)  = \sum _{\pi \in \DD_n^{<,\pm}} 
t^ {\altr_D(\pi)}=\sum_{i=1}^n R_{n,i}^{D,<,\pm} t^i.
\end{equation*}

Recall the sets $(\BB_n - \DD_n)^{>,\pm}$ and 
$(\BB_n-\DD_n)^{<.\pm}$
Lastly, define the appropriate polynomials 
enumerating $\altr_D$ over these sets.
\begin{eqnarray*}
R_n^{B-D,>,\pm} (t) & = & \sum _{\pi \in (\BB_n - \DD_n)^{>,\pm} } 
t^ {\altr_D(\pi)}= \sum_{i=1}^n R_{n,i}^{B-D,>,\pm} t^i,  \\
R_n^{B-D,<,\pm} (t) & = & \sum _{\pi \in (\BB_n - \DD_n)^{<,\pm}} 
t^ {\altr_D(\pi)}=\sum_{i=1}^n R_{n,i}^{B-D,<,\pm} t^i. 
 \end{eqnarray*}
We start by proving the following type D counterpart of Lemma
\ref{lem:invb_preserved_under_action}.
\begin{lemma}
\label{lem:invd_preserved_under_action}
For positive integers $n$, and $\pi \in \DD_n$, we have 
$\inv_D(\Sf_k(\pi)) \equiv \inv_D(\pi) $ (mod 2) whenever $k \in T_n$.  
\end{lemma}
\begin{proof}
Though $\inv_D(\pi)$ is defined only for $\pi \in \DD_n$,
$\inv_B(\pi)$ is defined even if $\pi \in \DD_n$ as $\DD_n \subseteq \BB_n$.
From \eqref{eqn:typeb-inv-defn} and \ref{eqn:inv-dtype-defn}, we get 
$\inv_B(\pi) = \inv_D(\pi) + |\Negs(\pi)|$.  Thus, 
$\inv_B(\pi) \equiv \inv_D(\pi)$ (mod 2) for $\pi \in \DD_n$.
The proof now follows from Lemma
\ref{lem:invb_preserved_under_action}.
\end{proof}

\begin{remark}
\label{rem:grou_actions}
By Lemma 
\ref{lem:invd_preserved_under_action}, under
the action of the group $\ZZ_2^m$ on $\DD_n^>$, all $\pi$ of any 
orbit $O$ have the same parity of 
$\inv_D(\pi)$ (mod 2) and thus, $O$
lies entirely in $\DD_n^{>,+}$ or entirely 
in $\DD_n^{>,-}$.  In a similar manner, it is 
easy to see 
that $\ZZ_2^m$ acts on each of the sets 
$R_n^{D,<,\pm}$, 
$R_n^{B-D,>,\pm}$ and $R_n^{B-D,<,\pm}$.  
\end{remark}

\begin{remark}
\label{rem:on-bona}
Lemma \ref{lem:invd_preserved_under_action} fails for a straightforward
extension of the  group action of B{\'o}na to the parity
based subsets of the type B and the type D cases.  
We 
illustrate this point for the type B case.  Similar
examples can be given for the type D case as well.
Given $\pi \in \BB_n$ and an index $i$, consider the sequence 
$S_i = \pi_i,\pi_{i+1},\ldots,\pi_n$ of elements of $\pi$.  In the 
sequence $S_i$, replace the $p$-th smallest element 
by the $p$-th largest element for all values of $p$.  
Let $c_i: \BB_n \mapsto \BB_n$ be
this operation.  Consider the group action generated by $c_i$ with
$i \in \{3,5,\ldots,t\}$ where $t = n-1$ if $n$ is even and 
$t = n-2$ if n is odd.  

This is a group action on $\BB_n^>$ but all orbits 
do not contain elements $\pi$ with the same parity of $\inv_B(\pi)$.
To see this, let $\pi = 1,2,\ol{3},\ol{4} \in \BB_4^>$.  Then,
$c_3(\pi) = 1,2,\ol{4},\ol{3}$.  Clearly, $\inv_B(\pi) \not\equiv
\inv_B(c_3(\pi))$ (mod 2) and so the orbit of $\pi$ lies
neither within $B_n^{>,+}$ nor $B_n^{>,-}$.
\end{remark}

Remark \ref{rem:grou_actions} helps us to prove the following refinement of 
Theorem 
\ref{thm:divisibilty_of_alternatingrunpolynomial_by_highpower_oneplust_b}.
\begin{theorem}
\label{thm:dbl_rstrct_div_d}
For positive integers $n$, the polynomials 
$R_n^{D,>,\pm} (t)$, $R_n^{D,<,\pm} (t)$, 
$R_n^{B-D,>,\pm} (t)$ and $R_n^{B-D,<,\pm}(t)$ 
are divisible by $(1+t)^m$ where $m= \nmhalf$.
\end{theorem}
\begin{proof}
We first consider the polynomial $R_n^{D,>,+} (t)$.  
By Remark \ref{rem:grou_actions},
$\ZZ_2^m$ acts on $\DD_n^{>,+}$.
As done in the proof of Theorem 
\ref{thm:divisibilty_of_alternatingrunpolynomial_by_highpower_oneplust_b},
summing $t^{\altr_D(\pi)}$ over the orbits, we get that $R_n^{D,>,+} (t)$ 
is divisible by $(1+t) ^{m}$. 
In an identical manner, one can prove that $(1+t)^m$ divides the
polynomials $R_n^{D,>,-} (t)$, $R_n^{D,<,\pm} (t)$, $R_n^{B-D,>,\pm} (t)$
and $R_n^{B-D,<,\pm} (t)$.  The proof is complete. 
\end{proof}

Lemma \ref{lem:chow-ma-powers-alt} and
Theorem 
\ref{thm:divisibilty_of_alternatingrunpolynomial_by_highpower_oneplust_b}
give a moment-type identity which is implicit in the work 
of Gao and Sun.  Combining Theorem 
\ref{thm:dbl_rstrct_div_d} and Lemma \ref{lem:chow-ma-powers-alt},
we get a the following refined moment-type identity. 

\begin{theorem}
	\label{thm:moment_id_dbl_rstrct_d}
	For $n \geq 2k+3$, we have
	\begin{eqnarray*}
		1^kR_{n,1}^{D,>, \pm}+3^kR_{n,3}^{D,>, \pm }+5^kR_{n,5}^{D,>,\pm}+\cdots 
		& = & 
		2^kR_{n,2}^{D,>,\pm}+4^kR_{n,4}^{D,>,\pm}+6^kR_{n,6}^{D,>,\pm}+\cdots,
		 \\
		1^kR_{n,1}^{D,<, \pm}+3^kR_{n,3}^{D,<, \pm }+5^kR_{n,5}^{D,<,\pm}+\cdots 
		& = & 
		2^kR_{n,2}^{D,<,\pm}+4^kR_{n,4}^{D,<,\pm}+6^kR_{n,6}^{D,<,\pm}+\cdots, \\
		1^kR_{n,1}^{B-D,>, \pm}+3^kR_{n,3}^{B-D,>, \pm }+\cdots 
		& = & 
		2^kR_{n,2}^{B-D,>,\pm}+4^kR_{n,4}^{B-D,>,\pm}+\cdots, 		\\
		1^kR_{n,1}^{B-D,<, \pm}+3^kR_{n,3}^{B-D,<, \pm }+\cdots 
		& = & 
		2^kR_{n,2}^{B-D,<,\pm}+4^kR_{n,4}^{B-D,<,\pm}+\cdots.
	\end{eqnarray*}
\end{theorem}

\input{longest}

\section*{Acknowledgement}
The first author acknowledges support from a CSIR-SPM
Fellowship.

\bibliographystyle{acm}
\bibliography{main}
\end{document}

%% file: longest.tex
\section{Longest alternating subsequence polynomials}
\label{sec:alt-seq-results}

For results involving $\SSS_n$, we start by defining some sets and 
their corresponding alternating run enumerating polynomials.  
We partition $\SSS_n$ into four (disjoint) subsets and compute the
signed alternating runs polynomial on each of these subsets. 
Our partitioning is based on the type of the first and 
the last pairs.  Either pair could be an ascent or a
descent.  When $n \geq 3$, we get the following four sets:
\begin{eqnarray*}
\SSS_{n,a,a} & = & \{\pi \in \SSS_n : \pi_1 < \pi_2, \pi_{n-1}< \pi_n \},  \hspace{5 mm}
\SSS_{n,a,d}= \{\pi \in \SSS_n : \pi_1 < \pi_2, \pi_{n-1}> \pi_n \}, \\
\SSS_{n,d,a} & = &  \{\pi \in \SSS_n : \pi_1 > \pi_2, \pi_{n-1}< \pi_n \},   \hspace{5 mm}
\SSS_{n,d,d}= \{\pi \in \SSS_n : \pi_1 > \pi_2, \pi_{n-1}> \pi_n \}. 
\end{eqnarray*} 
	
Define the following alternating runs enumerator polynomials.
	
	 \begin{enumerate}
		\item $\SR_{n,a,a}(t)=\sum_{\pi \in \SSS_{n,a,a}}(-1)^{\inv(\pi)}t^{\altr(\pi)},$
		\item $\SR_{n,a,d}(t)=\sum_{\pi \in \SSS_{n,a,d}}(-1)^{\inv(\pi)}t^{\altr(\pi)} ,$
		\item $\SR_{n,d,a}(t)=\sum_{\pi \in \SSS_{n,d,a}}(-1)^{\inv(\pi)}t^{\altr(\pi)} ,$
		\item $\SR_{n,d,d}(t)=\sum_{\pi \in \SSS_{n,d,d}}(-1)^{\inv(\pi)}t^{\altr(\pi)} ,$
		\item  $\SR_{n,a,-}(t)=\sum_{\pi \in \SSS_{n,a,a} \cup \SSS_{n,a,d} } (-1)^{\inv(\pi)}t^{\altr(\pi)}.$
	\end{enumerate}
	
The following result was proved in 
\cite[Corollary 22]{dey-siva-sgn-alt-runs-enum}.
	
\begin{theorem}
\label{thm:nmod4_uni}
For positive integers $k$, the following signed enumeration 
results hold.  When $n=4k$ and $n=4k+1$, we have:
\begin{eqnarray*}
\SR_{n,a,a}(t) & = &\SR_{n,d,d}(t)=t(1+t^2)(1-t^2)^{2k-2}, \\
\SR_{n,a,d}(t) & = & \SR_{n,d,a}(t)=-2t^2(1-t^2)^{2k-2}, \\
\SR_{n}(t) & = & 2t(1-t)^{2k}(1+t)^{2k-2}.
\end{eqnarray*}
	
When $n=4k+2$ and $n=4k+3$, we have:
\begin{eqnarray*}
\SR_{n,a,d}(t) & = & 0   \mbox{ and } \SR_{n,d,a}(t)=0,\\
\SR_{n,a,a}(t) & = & -\SR_{n,d,d}(t) = t(1-t^2)^{2k}, \\
\SR_{n}(t) & = & 0. 
\end{eqnarray*}
\end{theorem}

Using Theorem \ref{thm:nmod4_uni}, we get the following 
version of \eqref{eqn:reln_altruns-altseq} for $\AAA_n$ and 
$\SSS_n - \AAA_n$.

\begin{theorem}
\label{thm:reltn_altrun_altseq_typea_pm}
For positive integers $n=4k$ and $n=4k+1$, we have 
\begin{equation}
\label{eqn:rltnbetweenaltrunaltsubseq01mod4}
\AS_n^{\pm}(t)= \frac{(1+t)R_n^{\pm}(t)}{2}. 
\end{equation}
For positive integers $n=4k+2$ and $n=4k+3$, we have 
\begin{equation}
\label{eqn:rltnbetweenaltrunaltsubseq23mod4}
\AS_n^{\pm}(t)= \frac{(1+t)R_n^{\pm}(t) \pm t(1-t)(1-t^2)^{2k}}{2}. 
\end{equation}
\end{theorem}
	
\begin{proof}
Consider the map  $f: \SSS_n \mapsto \SSS_n$  defined as

$$f(\pi_1, \pi_2, \dots, \pi_n)= n+1-\pi_1, n+1-\pi_2, \dots, n+1-\pi_n.$$
		
As the map $f$ flips the parity of $\inv$ if and only if 
$n \equiv 0,1$ (mod 4), we have 
\begin{eqnarray}
\SAS_n(t) = \begin{cases}
\displaystyle (1+t)\SR_{n,a,-}(t)& \text {when $n \equiv 0,1$ (mod 4)}.\\
\displaystyle (1-t)\SR_{n,a,-}(t)& \text {when $n \equiv 2,3$ (mod 4)}.
\end{cases}
\end{eqnarray}
		
Therefore, when $n=4k$ or $n=4k+1$, we have 
\begin{eqnarray*}
2\AS_n^{\pm}(t) & = &\AS_n(t) \pm \SAS_n(t) \\
& = & \frac{1}{2}(1+t)R_n(t) \pm (1+t)\SR_{n,a,-}(t)\\
& = &  \frac{1}{2}(1+t)R_n(t) \pm \frac{1}{2}(1+t)\SR_{n}(t) = (1+t)R_n^{\pm}(t).
\end{eqnarray*}
The second line above uses \eqref{eqn:reln_altruns-altseq}. 
Therefore, we are done in this case. When $n=4k+2$ or $n=4k+3$, we have  
\begin{eqnarray*}
2\AS_n^{\pm}(t) & = &\AS_n(t) \pm \SAS_n(t) \\
& = & \frac{1}{2}(1+t)R_n(t) \pm (1-t)\SR_{n,a,-}(t)\\
& = &  (1+t)R_n^{\pm}(t) \pm t(1-t)(1-t^2)^{2k}.
\end{eqnarray*}
The last line uses  
\cite[Corollary 22]{dey-siva-sgn-alt-runs-enum}.
The proof is complete. 
\end{proof}
	
We need the following result proved in 
\cite[Theorem 3]{dey-siva-sgn-alt-runs-enum}.
\begin{theorem}
\label{thm:Rnplusminust_is_divible_by_1plust_m}
For positive integers $n \geq 4$, let $m = \floor{(n -2)/2}$.
When $n \equiv 0,1$ (mod 4), the 
polynomials $R_n^{\pm}(t)$ are divisible by  $(1 + t)^{m-1}$, 
When $n \equiv 2,3$ (mod 4), the polynomials 
$R_n^{\pm}(t)$ are divisible by $(1 + t)^m$.
\end{theorem}

{\bf Proof of Theorem \ref{thm:divisibility_altseq_typea_pm} :}
Follows immediately from Theorem 
\ref{thm:reltn_altrun_altseq_typea_pm} and Theorem 
\ref{thm:Rnplusminust_is_divible_by_1plust_m}.
{\hspace*{\fill}{\eod}}
	
We move on to our final proof of this work.

{\bf Proof of Theorem \ref{thm:divisibility_altseq_typebd} :}
Consider the type B case first.
Using the map $\Sf_1$, it is easy to see that
$$\AS^B_n(t)= (1+t) \sum _{\pi \in \BB_n^>} t^{\as_B(\pi)}= 
(1+t)R_n^{B,>}(t).$$ 
The proof is complete by combining with Theorem
\ref{thm:divisibilty_of_alternatingrunpolynomial_by_highpower_oneplust_b}.
We give our proof for $\AS_n^{B,\pm}(t)$ below.

For the type D case, recall the following sets from 
Section \ref{Sec:proofofmainthm}:  let
$T_n= \{1,3,5,\dots,n-1\}$ when $n$ is even and 
let $T_n= \{2,4,6,\dots,n-1\}$ when $n$ is odd. 
As seen in Section  \ref{Sec:proofofmainthm},  
we have a group action on
$\DD_n$.  We can repeat the same argument to get 
that $\AS_n^D(t)$ is divisible by $(1+t)^{|T_n|}$.
As $|T_n|= \floor{n/2}$, we are done.  
Further, by Lemma \ref{lem:invb_preserved_under_action} and
Lemma \ref{lem:invd_preserved_under_action}, since these 
are actions on $\DD_n^{\pm}$ and on $\BB_n^{\pm}$, an 
identical argument shows that the polynomials
$\AS_n^{B, \pm}(t)$ and $\AS_n^{D, \pm}(t)$  are divisible
by $(1+t)^{\floor{\frac{n}{2}}}$, completing the proof.
{\hspace*{\fill}{\eod}}

\begin{remark}
\label{rem:type-b-example}
The difference in the exponent of $(1+t)$ between the whole 
group and its positive-and-negative parts in Theorem
\ref{thm:divisibility_altseq_typea_pm} and Theorem
\ref{thm:divisibility_altseq_typebd}
exists and our result is the best that one can get.  
We give examples below.  For the $\SSS_n$ case, 
when $n=6$, one can check that $(1 + t)^3$ divides
$\AS_6(t)$, while $(1+t)^2$
divides $\AS^{\pm}_6(t)$ as we have 
\begin{eqnarray*}
\AS_6(t)  
 & = & (t+1)^3  (61t^3 + 28t^2 + t ) , \\
\AS^{+}_6(t) 
& = & (t + 1)^2t  (6t^2 + 8t + 1)(5t + 1  ), \\
\AS^{-}_6(t) 
& = & (t + 1)^2t^2(31t^2 + 43t + 16).
\end{eqnarray*}
For the type B case, when $n=3$, one can check that $(1 + t)^2$ 
divides $\AS^B_3(t)$, while $(1+t)$ divides
 $\AS^{B,\pm}_3(t)$ as we have
\begin{eqnarray*}
\AS^B_3(t)   = &  11t^4+23t^3+13t^2+t &  = (t+1)^2 t (11t + 1), \\
\AS^{B,+}_3(t)  = & 6t^4+11t^3+6t^2+t &  =  (t + 1)t(2t + 1)(3t + 1), \\
\AS^{B,-}_3(t)  = &  5t^4+12t^3+7t^2 & =  (t + 1)t^2(5t + 7).
\end{eqnarray*}

Finally, we also mention that moment identities similar to Theorem
\ref{thm:moment_id_dbl_rstrct_b} and Theorem
\ref{thm:moment_id_dbl_rstrct_d} can be given involving 
the coefficients of $\AS_n^{\pm}(t)$, $\AS_n^{B,\pm}(t)$ and 
$\AS_n^{D,\pm}(t)$
\end{remark}
